\documentclass{amsart}
\usepackage{amsmath}
\usepackage{amssymb}
\usepackage[dvipdfmx]{graphicx}
\usepackage{amscd}
\usepackage{amsthm}
\usepackage{fullpage}
\usepackage{mathrsfs}
\usepackage{setspace}
\usepackage{mathtools}
\usepackage{xcolor}
\usepackage[all]{xy}
\usepackage{url}
\usepackage{bbm}
\usepackage[colorlinks=true,linkcolor=blue,pagebackref]{hyperref}
\usepackage[poorman]{cleveref}
\usepackage[lite,abbrev,alphabetic]{amsrefs}
\newtheorem{dfn}{Definition}[section]
\newtheorem{thm}[dfn]{Theorem}
\newtheorem{lem}[dfn]{Lemma}
\newtheorem{prop}[dfn]{Proposition}

\newtheorem{ass}{Assumption}
\newtheorem{cond}[ass]{Condition}
\newtheorem{conj}[dfn]{Conjecture}
\numberwithin{equation}{section}

\def\C{\mathbb{C}}
\def\R{\mathbb{R}}
\def\Q{\mathbb{Q}}
\def\Z{\mathbb{Z}}
\def\A{\mathbb{A}}

\def\GL{\mathrm{GL}}
\def\PGL{\mathrm{PGL}}
\def\SL{\mathop{\mathrm{SL}}}

\def\SO{\mathop{\mathrm{SO}}}

\def\d{\,\mathrm{d}}

\def\bs{\backslash}

\def\1{\mathbf{1}}
\def\0{\mathbf{0}}

\def\bsl{\backslash}
\def\inf{\infty}

\def\fo{\mathfrak{o}}

\def\fO{\mathfrak{O}}

\def\fT{{\mathfrak{T}}}

\newcommand{\cE}{\mathcal{E}}
\newcommand{\cF}{\mathcal{F}}

\newcommand{\cH}{\mathcal{H}}

\newcommand{\cL}{\mathcal{L}}

\newcommand{\cO}{\mathcal{O}}
\newcommand{\cP}{\mathcal{P}}
\newcommand{\cQ}{\mathcal{Q}}

\newcommand{\cT}{\mathcal{T}}

\newcommand{\cW}{\mathcal{W}}

\DeclareMathOperator{\vol}{vol}

\DeclareMathOperator{\Hom}{Hom}
\DeclareMathOperator{\Emb}{Emb}

\DeclareMathOperator{\disc}{disc}
\DeclareMathOperator{\level}{level}
\DeclareMathOperator{\sgn}{sgn}
\DeclareMathOperator{\Nm}{Nm}
\DeclareMathOperator{\Tr}{Tr}
\DeclareMathOperator{\Cl}{Cl}

\DeclareMathOperator{\St}{St}

\def\trep{\mathbbm{1}}

\def\bb{\mathbf{b}}
\def\uG{\underline{G}}
\def\Emb{\mathrm{Emb}}
\def\Orb{\mathrm{Orb}}

\def\Tr{\mathop{\mathrm{Tr}}}
\def\Nm{\mathop{\mathrm{Nm}}}

\def\Hom{\mathop{\mathrm{Hom}}}

\def\sgn{\mathop{\mathrm{sgn}}}

\def\vol{\mathop{\mathrm{vol}}}

\def\disc{\mathrm{disc}}

\def\new{\mathrm{new}}
\def\eps{\varepsilon}
\def\St{\mathrm{St}}

\def\alg{\varphi}

\begin{document}

\title{Non-vanishing theorems for prime twists of some modular $L$-functions}

\author{Masataka Chida}
\author{Satoshi Wakatsuki}
\date{\today}

\begin{abstract}
In this paper, we give some non-vanishing results on the central values of prime twists of modular $L$-functions by imaginary quadratic fields for specific elliptic modular forms.
In particular, we show that the central values of prime twists of $L$-functions of some elliptic modular forms are always non-vanishing whenever the root number is positive. 
\end{abstract}

\maketitle

\setcounter{tocdepth}{1}

\tableofcontents

\section{Introduction}
Let $f$ be a newform of even weight $k$ and level $\Gamma_0(N)$.
Let $\pi$ be the cuspidal automorphic representation of $\mathrm{PGL}_2(\mathbb{A})$
generated by $f$, where $\mathbb{A}$ is the ring of ad\`eles of $\mathbb{Q}$.
For a quadratic field $E$ with the discriminant $\Delta_E$, we denote the quadratic character corresponding to
the quadratic extension $E\slash \mathbb{Q}$ by $\eta_E$.
There are a lot of important works concerning on the non-vanishing of the central values of quadratic twists of automorphic $L$-functions $L(\frac{1}{2} , \pi \otimes \eta_E)$.
In particular, Ono and Skinner \cite{OS} proved that
$$
\#\left\{ E \, \middle| \, | \Delta_E|\leq X,\,  L(\tfrac{1}{2} , \pi \otimes \eta_E) \neq 0 \right\} \gg \frac{X}{\log X}.
$$
The Goldfeld's conjecture \cites{OS,K,SWY2} expects
$$\displaystyle
\lim_{X\to \infty}\frac{\#\left\{ E \, \middle| \, | \Delta_E|\leq X, \, L(\frac{1}{2} , \pi \otimes \eta_E) \neq 0\right\} }{\# \{ E \mid |\Delta_E| \leq X\}  }=\frac{1}{2},
$$
and a weaker version of the conjecture (the weak Goldfeld conjecture) is stated as
$$
\#\left\{ E \, \middle| \, | \Delta_E|\leq X, \, L(\tfrac{1}{2} , \pi \otimes \eta_E) \neq 0\right\} \gg X.
$$
This conjecture has been especially considered when $f$ is corresponding to an elliptic curve over $\mathbb{Q}$. It is known that the weak Goldfeld conjecture holds for elliptic curves having a rational $3$-isogeny \cites{J,KL}.
For higher weight modular forms, Kohnen \cite{K} proved that there exists a normalized Hecke eigenform $f$ of level $1$ and weight $4k$ with $k \geq 3$ satisfying the weak version of the conjecture.
In particular, he showed the weak Goldfeld conjecture for
Ramanujan's $\Delta$-function $\Delta (z)=q\prod_{n=1}^\infty(1-q^n)^{24}$.
This result is partially generalized by Makiyama \cite{M}.
More strongly, the following non-vanishing property is expected.
\begin{conj}[\cites{CKRS, MO}]
If $f$ is a newform of level $\Gamma_0(N)$ and even weight $k$ with $k\geq 6$, then there are at most finitely many quadratic fields $E$
such that $\varepsilon (\frac{1}{2} , \pi \otimes \eta_E)=1$ and $L(\frac{1}{2},\pi\otimes \eta_E)= 0$, where $\varepsilon (\frac{1}{2} , \pi \otimes \eta_E)$ is the root number associated to $\pi \otimes \eta_E$.
\end{conj}
However, this conjecture is still open and even any examples of a modular form satisfying the conjecture are not known.
As a weak form of the conjecture, one can expect that there are at most finitely many quadratic fields $E$
such that $|\Delta_E|$ is a prime, $\varepsilon (\frac{1}{2} , \pi \otimes \eta_E)=1$ and $L(\frac{1}{2},\pi\otimes \eta_E)= 0$.
As a related result, Kohnen and Zagier \cite{KZ} proved that $L(\frac{1}{2},\pi\otimes \eta_E)\neq 0$
when $\Delta_E=p$ is a prime with $p\equiv 5 \mod 8$ and $f$ is the Ramanujan's $\Delta$-function.

Our first result supports this conjecture.
\begin{thm}\label{thm:1}
Let $f$ be the newform of weight $8$ and level $\Gamma_0(2)$ which is explicitly given by
$f(z)=\eta(z)^8\eta(2z)^8$, where $\eta (z)$ is the Dedekind $\eta$-function.
Let $\pi$ be the cuspidal automorphic representation of $\mathrm{ PGL}_2(\mathbb{A})$
generated by $f$.
Then for any imaginary quadratic field $E$ such that $-\Delta_E$ is a prime and $\varepsilon (\frac{1}{2}, \pi \otimes \eta_E)=1$,
we have $L(\frac{1}{2},\pi \otimes \eta_E)\neq 0$.
\end{thm}
The second theorem obtains a similar result for a weight 4 newform.
It is interesting to compare the following result with the expectation in \cite{CKRS}*{(26)}
\begin{thm}\label{thm:2}
Let $f$ be the newform of weight $4$ and level $\Gamma_0(6)$ which is explicitly given by
$f(z)=\eta (z)^2 \eta(2z)^2\eta (3z)^2 \eta (6z)^2 $.
Let $\pi$ be the cuspidal automorphic representation of $\mathrm{PGL}_2(\mathbb{A})$
generated by $f$.
Then for any imaginary quadratic field $E$ such that $-\Delta_E$ is a prime and $\varepsilon (\frac{1}{2}, \pi \otimes \eta_E)=1$,
we have $L(\frac{1}{2},\pi \otimes \eta_E)\neq 0$.
\end{thm}

Here we recall the conjecture by Conrey, Keating, Rubinstein and Snaith on the number of quadratic twists of modular $L$-functions whose central value vanishes.

\begin{conj}[\cites{CKRS,CKRS06}]
\begin{enumerate}
\item Let $f$ be a newform of weight $2$ with rational Fourier coefficients (i.e. a newform corresponding to an elliptic curve over $\mathbb{Q}$) and $\pi$ the cuspidal automorphic representation of $\mathrm{PGL}_2(\mathbb{A})$ generated by $f$. Then there is a constant $c_f^\pm \geq 0$ such that
$$
\# \left\{ E \, \middle| \, |\Delta_E| \textup{ : prime}, \, 0< \pm \Delta_E <X, \, \varepsilon (\tfrac{1}{2}, \pi\otimes \eta_E)=1, \, L(\tfrac{1}{2}, \pi\otimes \eta_E)= 0  \right\} \sim c_f^\pm X^{3\slash 4}(\log T)^{-5\slash 8}.
$$
\item 
Let $f$ be a newform of weight $4$ with rational Fourier coefficients and $\pi$ the cuspidal automorphic representation of $\mathrm{PGL}_2(\mathbb{A})$ generated by $f$. Then there is a constant $c_f^\pm \geq 0$ such that
$$
\# \left\{ E \, \middle| \, |\Delta_E| \textup{ : prime}, \, 0< \pm \Delta_E <X, \, \varepsilon (\tfrac{1}{2}, \pi\otimes \eta_E)=1, \, L(\tfrac{1}{2}, \pi\otimes \eta_E)= 0  \right\} \sim c_f^\pm X^{1\slash 4}(\log T)^{-5\slash 8}.
$$
\end{enumerate}
\end{conj}

We remark that the constant $c_f^\pm$ in the above conjecture can be zero.
For the weight $2$ case, Delaunay \cite{Delaunay} gave an example of weight $2$ newform $f$ with $c_f^-=0$.
Similarly, Theorem \ref{thm:2} can be interpreted as an example of weight $4$ newform $f$ with $c_f^-=0$.

Our method of the above results is based on the explicit computation of toric periods of automorphic forms.
Here we briefly explain our strategy.
Let $D$ be a division quaternion algebra over $\mathbb{Q}$
and $\mathcal{O}$ an Eichler order of $D$.
We assume that the class number of $\mathcal{O}$ is one.
Then the space of modular forms on $D$ of level $\mathcal{O}$ has a simple description in terms of harmonic polynomials. 
Let $\pi$ be an irreducible automorphic representation of $( D\otimes_{\mathbb{Q}} \mathbb{A})^\times \slash \mathbb{A}^\times$.
Let $\pi'$ be the cuspidal automorphic representation of $\mathrm{PGL}_2(\mathbb{A})$ corresponding to $\pi$ via Jacquet-Langlands correspondence.
Let $E$ be a quadratic field such that there is an embedding $\iota_0 : E \hookrightarrow D$.
Let $T$ be the torus associated to $\iota_0$.
Then we can define the torus period on $\pi$ with respect to $E$ by
$$
\mathcal{P}_{\iota_0,E}(\phi)=\int_{T(\mathbb{Q})\backslash T(\mathbb{A})}\phi (t)\d t
$$
for $\phi \in \pi$.
By Waldspurger's result, the non-vanishing of the toric period on $\pi$ with respect to $E$ is equivalent to the non-vanishing of $L$-value $L(\frac{1}{2},\pi')L(\frac{1}{2},\pi' \otimes \eta_E)$ under a certain condition on local root numbers (see Theorem \ref{thm:Waldspurger} for details). Hence it is enough to show that the toric period is non-zero on $\pi$.
Since we assumed that the class number of $\mathcal{O}$ is one, we have $G_{\mathbb{A}}=GK$, where $G=D^\times \slash \mathbb{Q}^\times$ and $K$ is the open compact subgroup of $G_{\mathbb{A}}$ determined by $\mathcal{O}$.
Then the space of algebraic modular forms on $G$ of level $\Gamma = G \cap K$ associated with the $(2l+1)$-dimensional irreducible representation of $G_\R$ is identified with the space of $\Gamma$-invariant harmonic polynomials $\mathcal{H}_l^\Gamma$ of degree $l$.
Using this description, we can construct a specific automorphic form $\phi \in \pi$ such that the toric period of $\phi$ is represented by a sum of the values of the harmonic polynomial over CM points (Lemma \ref{lem:periodformula}).
Let $\pi'$ be the automorphic representation generated by $f=f_1$ in Theorem \ref{thm:1}. Then we will show that the quotient of the period is an integer and it is congruent to the class number $h_E$ of $E$ modulo $2$ if $E$ is an imaginary quadratic field such that the discriminant $\Delta_E$ satisfies $\Delta_E \equiv 1 \mod 4$ (Lemma \ref{lem:d4sum}). The computation of root numbers shows that the condition $\varepsilon (\frac{1}{2}, \pi' \otimes \eta_E)=1$ is equivalent to the condition $\Delta_E \equiv 5 \mod 8$. Hence the non-vanishing of the toric period can be reduced to the non-vanishing of the class number $h_E\mod 2$ when the root number is positive.
Then Theorem \ref{thm:1} follows from the genus theory (see the proof of Theorem \ref{thm:disc2}).

The key ingredient of the proof is to find a congruence between the values of the modular form at CM points and the discriminant of the imaginary quadratic field $E$ (see the proof of Lemma \ref{lem:d4sum}).
In particular, we can show that the values of the modular form at CM points are constant modulo $2$ in this case.
The proof of the congruence is elementary, however it is ad hoc. It is not clear that such a congruence always exists.
Hence it may be an interesting problem to consider when such a congruence exists.

To describe the congruence more precisely, we introduce some notations.
Denote $V_D=\{ x\in D \mid \operatorname{Tr}(x)=0\}$, which is a $3$-dimensional $\mathbb{Q}$-vector space. Let $L(\mathcal{O})$ be the lattice in $V_D$ defined by $L(\mathcal{O})=\{ x\in \mathbb{Z}+2\mathcal{O} \mid \mathrm{Tr}(x)=0\}$. We fix an isomorphism $\cT : V_D \overset{\sim}{\longrightarrow}\mathbb{Q}^3$.
Let $\mathfrak{o}_E$ be the ring of integers of $E$ and $\operatorname{Emb}(\mathfrak{o}_E,\mathcal{O})$ the set of optimal embeddings. We fix an optimal embedding $\iota_0 \in \operatorname{Emb}(\mathfrak{o}_E,\mathcal{O})$ and consider the orbit $\operatorname{Orb}(\iota_0)=\{ \iota_1 ,\ldots , \iota_{h_E}\}$ of $\iota_0$ by the ideal class group $\operatorname{Cl}(\mathfrak{o}_E)$.
Then one can take $\{ a_j \}_{1\leq j \leq h_E} \in V_D$ such that $\operatorname{Nm}(a_j)=-\Delta_E$ and $\mathbb{Z}a_j=\iota_j(E)\cap L(\mathcal{O})$.
For each $a \in V_D$ and $\varphi \in \mathcal{H}_l^\Gamma$,
we can associate an automorphic form $\phi=\mathcal{L}_a(\varphi)$ (see the equation (\ref{eqn:lift}) for the precise definition).
Then the toric period of $\phi$ is given by
$$
\mathcal{P}_{\iota_0,E}(\mathcal{L}_{a_0}(\varphi))=c \sum_{j=1}^{h_E}\varphi \circ \mathcal{T}(a_j),
$$
where $c$ is a nonzero constant (Lemma \ref{lem:periodformula}).
Now we consider the case that $\operatorname{disc}(D)=2$ and $\operatorname{level}(\mathcal{O})=1$.
Suppose $l=3$. Then we can find a harmonic polynomial $\varphi \in \mathcal{H}_3^\Gamma$ of degree 3 explicitly given by
$$
\varphi (x)=x_1^3+x_2^3+x_3^3-x_1^2x_2-x_1^2x_3-x_2^2x_1-x_2^2x_3-x_3^2x_1-x_3^2x_2+2x_1x_2x_3.
$$
We set $(a_{j1},a_{j2},a_{j3})=\mathcal{T}(a_j)$.
Since we have $\mathrm{Nm}(a_j)=-\Delta_E$ by Lemma \ref{lem:A},
one has $3a_{j1}^2+3a_{j2}^2+3a_{j3}^2-2a_{j1}a_{j2}-2a_{j2}a_{j3}-2a_{j3}a_{j1}=-\Delta_E$.
Using this relation it is easy to see
$$
(a_{j1}+a_{j2}+a_{j3})^2\equiv \Delta_E\equiv 1 \mod 4, \quad \text{hence} \quad  a_{j1}+a_{j2}+a_{j3} \equiv \pm 1 \mod 4 
$$
(see the proof of Lemma \ref{lem:d4sum} for details).
Hence we have $\varphi \circ \mathcal{T}(a_j)\equiv(a_{j1}+a_{j2}+a_{j3})^3\equiv \pm 1 \mod 4$.
Therefore this congruence gives
$$
\sum_{j=1}^{h_E}\varphi \circ \mathcal{T}(a_j)\equiv h_E \mod 2.
$$
This implies $$
\mathcal{P}_{\iota_0,E}(\mathcal{L}_{a_0}(\varphi))\neq 0
$$
when $\Delta_E$ is a prime and $\Delta_E \equiv 5 \mod 8$, since $h_E$ is odd under this condition.
Hence we have Theorem \ref{thm:1}.

Moreover, we could find similar congruence relations in the following cases:
\begin{enumerate}
\item mod $6$ congruence relation for the newform of weight $4$, level $\Gamma_0(6)$ (this gives Theorem \ref{thm:2}).
\item mod $4$ congruence relation for the newform of weight $6$, level $\Gamma_0(6)$ (see \S 7.1).
\item mod $6$ congruence relation for the newform of weight $6$, level $\Gamma_0(3)$ (see \S 7.2).
\item mod $10$ congruence relation for the newform of weight $4$, level $\Gamma_0(5)$ (see \S 7.3).
\end{enumerate}

We note that it is important to consider a $\mathbb{Q}$-rational structure of the space of modular forms to see such a congruence. In particular, one needs a $\mathbb{Q}$-rational representation on the space of harmonic polynomials.
For this purpose, we will use modular forms on orthogonal groups. 

This paper is organized as follows.
In \S \ref{sec:2}, we review some basic facts of the theory of algebraic modular forms on quaternion algebra that will be needed later.
In \S \ref{sec:3}, we recall Waldspurger's result on toric periods.
Using optimal embeddings, we also give a more precise formula for toric periods.
After recalling the Jacquet-Langlands correspondence and some basic facts on root numbers in \S \ref{sec:JL},
we give the proofs of Theorem \ref{thm:1} and \ref{thm:2} in \S \ref{sec:5} and \S \ref{sec:6} respectively.
Finally, we give further examples of non-vanishing theorem in \S \ref{sec:7}.

\medskip

\noindent
\textbf{Acknowledgments.} 
The authors would like to thank Ken Ono for helpful discussions.
The authors are grateful to John B. Conrey, Jonathan P. Keating, Michael O. Rubinstein and Nina C. Snaith for informing a paper of Delaunay and explaining the relation between their conjecture and our result. 
The first author is partially supported by JSPS Grants-in-Aid for Scientific Research Grant Numbers 22H00096, 23K03038.
The second author is partially supported by JSPS Grants-in-Aid for Scientific Research Grant Numbers 20K03565, 21H00972.

\section{Preliminaries}\label{sec:2}

The cardinality of a finite set $A$ is denoted by $|A|$ or $\#A$.
We denote by $\A$ the ring of ad\`eles of $\Q$.
Let $\A_f$ denote the ring of finite ad\`eles.
Finite places of $\Q$ are identified with primes.
We write the real place of $\Q$ as $\infty$.
Let $\Q_v$ denote the completion of $\Q$ at a place $v$.

\subsection{Quaternion algebra $D$}

Let $D$ be a division quaternion algebra over $\Q$. 
We denote by $\disc(D)$ its discriminant. 
For a place $v$ of $\Q$,  set $D_v=D\otimes_\Q\Q_v$. 
We also set $D_\A=D\otimes_\Q\A$ and $D_{\A_f}=D\otimes_\Q\A_f$.
Choose a standard $\Q$-involution $\sigma$ on $D$, and set $\Nm(x)=x\, \sigma(x)$ and $\Tr(x)=x+ \sigma(x)$ for $x\in D$. 
Throughout this paper, we assume the following condition:
\begin{ass}\label{ass1}
$D$ is definite, that is, $D_\infty$ is isomorphic to the Hamilton's quaternion.
\end{ass}

Let $\cO$ be an Eichler order in $D$.
Denote its level by $\level(\cO)$. 
The discriminant of $\cO$ is defined as $\disc(\cO)=\disc(D)\level(\cO)$.
For a prime $p$,  set $\cO_p=\cO\otimes_\Z\Z_p$.
Then, we have $\cO_p\cong M_2(\Z_p)$ if $p\nmid\disc(\cO)$, $\cO_p$ is the maximal order of $D_p$ if $p\mid \disc(D)$, and $\displaystyle \cO_p\cong \left\{ \left(\begin{smallmatrix} a&b \\ c&d \end{smallmatrix}\right) \in M_2(\Z_p)\; \middle| \;  c\equiv 0 \mod \level(\cO)\Z_p \right\}$ if $p\mid \level(\cO)$.

\subsection{Algebraic group $\uG$}

We denote by $\uG$ the algebraic group over $\Q$ such that $\uG(R)=(D\otimes_\Q R)^\times/R^\times$ for any $\Q$-algebra $R$.
To simplify the notation,  we write the group of $R$-rational points $\uG(R)$ as $G_R$ for any $\Q$-algebra $R$, and we also write the group $\uG(\Q)$ as $G$. 
It is known that $G$ is identified with the automorphism group of $D$ via conjugation. 
For a place $v$ of $\Q$,  set $G_v=\uG(\Q_v)$. 
We use similar notation for subgroups defined over $\Q$. 
The projection map $D_\A^\times\rightarrow G_\A$ is written as $x\mapsto\bar{x}$.
We denote the image of a subset $Y$ of $D_\A^\times$ by $\overline{Y}$.

For each place $v$,  let $K_v$ be the open compact subgroup of $G_v$ given by $K_v=\overline{(\cO_v^\times)}$ if $v$ is finite and $K_\infty=G_\inf$.
Then $K=\prod_vK_v$ (resp. $K_f=\prod_{v<\inf}K_v$) is an open compact subgroup of $G_\A$ (resp. $G_{\A_f}$).


Throughout this paper, we also assume the following:
\begin{ass}\label{ass2}
The class number of $\cO$ is $1$, that is, we have $G_\A=G K$.
\end{ass}

The list of Eichler orders of class number one is as follows:
\[
(\disc(D),\level(\cO)) = \begin{array}{l} (2,1),  \;\; (2,3), \;\; (2,5), \;\; (2,11), \;\; (2,9), \\ (3,1), \;\; (3,2), \;\; (3,4), \\ (5,1), \;\; (5,2), \\ (7,1), \\ (13,1) \end{array}
\]
cf. Voight's homepage \cite{Voight2}.

\subsection{Harmonic polynomials} 

Set $V_D=\{x\in D \mid \Tr(x)=0\}$. 
Then, $V_D$ is a $3$-dimensional vector space over $\Q$, and a quadratic form $Q:V_D\times V_D\to \Q$ is defined by
\[
 Q(x,y)= \frac{1}{2} \left\{ \Nm(x+y)-\Nm(x)-\Nm(y) \right\} \qquad (x,y\in V_D).
\]
The group $G$ acts on $V_D$ as $x\cdot g=g^{-1}xg$ $(x\in V_D, \;\; g\in G)$. 
Since $\Nm(x\cdot g)=\Nm(x)$, we obtain the isomorphism
\[
G\cong \SO(V_D,Q)=\{ g\in \SL(V_D) \mid Q(x\cdot g,y\cdot g)=Q(x,y) \;\; \textup{for all } x,y\in V_D\}.
\]
Fix a basis $\bb_1$, $\bb_2$, $\bb_3$ of $V_D$, and we define the isomorphism $\cT : V_D \stackrel{\sim}{\longrightarrow} \Q^3$ by 
\[
\cT(a_1\bb_1+a_2\bb_2+a_3\bb_3)=(a_1,a_2,a_3) .
\]
A symmetric matrix $\cQ$ of degree $3$ is defined by $\cQ=(Q(\bb_i,\bb_j))_{1\le i,j\le 3}$, and set 
\[
\SO(\cQ)=\{g\in\mathrm{SL}_3 \mid g\cQ\, {}^t\!g=\cQ\}.
\]
A morphism $\cF: \, G \to \SO(\cQ)$ over $\Q$ is given by
\[
\cT(g^{-1} x g) = \cT(x)\,  \cF(g), \qquad x\in V_D, \quad g\in G.
\]
We let $\cW_l=\{ \alg\in \Q[x_1,x_2,x_3] \mid \alg$ is homogeneous and $\deg \alg=l\}$. 
A Laplace operator $\Delta_\cQ:\cW_l \to \cW_{l-2}$ is defined by
\[
\Delta_\cQ=\sum_{1\le i,j\le 3} \mathfrak{q}_{ij}\, \frac{\partial^2}{\partial x_i \, \partial x_j} ,\quad \textup{where } \cQ^{-1}=(\mathfrak{q}_{ij})_{1\le i,j\le 3}.
\]
A vector space $\cH_l$ of harmonic polynomials of degree $l$ is the kernel of $\Delta_l$, that is,
\[
\cH_l=\{ \alg\in \cW_l \mid \Delta_\cQ \alg=0\}.
\]
The group $G$ acts on $\cW_l$ by $(g\cdot \alg)(x)=\alg(x\cF(g))$ $(f\in\cW_l$, $g\in G)$, and the action of $G$ commutes with $\Delta_\cQ$. 
Hence, a $\Q$-rational representation $\rho_l$ of $G$ on $\cH_l$ can be defined by
\[
(\rho_l(g)\alg)(x_1,x_2,x_3)=\alg((x_1,x_2,x_3)\cF(g)) ,\quad \alg\in \cH_l,\;\; g\in G.
\]
It is known that $\dim \cH_l=2l+1$ and $\rho_l$ is irreducible. 
Each polynomial $\alg$ in $\cH_l$ is identified with the polynomial function as $\Q^3\ni a\mapsto \alg(a) \in \Q$. 
Then, a $\Q$-rational representation $\varrho_l$ of $G$ on $H_l= \{ \alg\circ \cT \mid \alg\in \cH_l\}$ is given by
\[
(\varrho_l(g)\cdot\psi)(x)=\psi(x\cdot g) \qquad (\psi\in H_l, \;\; g\in G, \;\; x\in V_D),
\]
and we have $\varrho_l(g)(\alg\circ \cT)=(\rho_l(g)\alg)\circ \cT$.

\subsection{Algebraic modular forms}

Let $\cO^\times$ denote the group of unit elements in $\cO$. 
By Assumption \ref{ass1}, $\#\cO^\times$ is finite and we have $\cO^\times=\{ x\in\cO \mid \Nm(x)=1\}$.

Let $\Gamma=G\cap  K=\overline{\cO^\times}$. Then, $\Gamma$ is a finite subgroup of $G$. 
Set
\[
\cH_l^\Gamma=\{ \alg\in \cH_l \mid \rho_l(\gamma)\alg=\alg \, \, \textup{for all }\gamma\in\Gamma \}, \quad H_l^\Gamma=\{ \alg\circ \cT \mid \alg\in\cH_l^\Gamma\}.
\]
Any function $\psi\in H_l^\Gamma$ satisfies $\varrho_l(\gamma)\psi=\psi$ for all $\gamma\in\Gamma$.
We identify $\cH_l^\Gamma$ with $H_l^\Gamma$ by $\alg\mapsto \alg\circ\cT$. 
The functions in $\cH_l^\Gamma\otimes\C$ and $H_l^\Gamma\otimes\C$ are called algebraic modular forms.

\subsection{Lift $\cL_a$ from $\cH_l^\Gamma$ to $L^2(G\bsl G_\A/ K_f)$}\label{sec:lift}

Fix an element $a\in V_D$ and set $\alpha=\cT(a)\in \Q^3$. 
By Assumption \ref{ass1}, for each $\alg\in\cH_l^\Gamma\otimes\C$, the function
\[
G_\inf\ni g_\inf \mapsto \alg(\alpha\, \cF(g_\inf^{-1})) \in \C
\]
belongs to $L^2(\Gamma\bsl G_\inf)$. 
By Assumption \ref{ass2}, we have the diffeomorphism $\Gamma\bsl G_\inf \cong G\bsl G_\A / K_f$. 
Hence, a linear mapping $\cL_a: \cH_l^\Gamma\otimes\C \to L^2(G\bsl G_\A/K_f)$ is defined by
\begin{equation}\label{eqn:lift}
\cL_a(\alg)(\gamma g_\inf k_f)=\alg(\alpha\, \cF(g_\inf^{-1})) \qquad (\alg\in\cH_l, \;\; g_\inf\in G_\inf, \;\; \gamma\in G, \;\; k_f\in K_f).
\end{equation}
When $\alpha\neq (0,0,0)$, we note that $\alg\not\equiv 0$ if and only if $\cL_a(\alg)\not\equiv 0$. 
Hence, $\cL_a$ is injective if $\alpha\neq (0,0,0)$.

\subsection{Hecke operators}

For a prime $p$ with $p\nmid \disc(\cO)$, 
we set
\[
\fT_p=\{\overline{x} \mid x\in\cO,\;\;  \Nm(x)= p \} \subset G.
\]
Since $\fT_p$ is bi-$\Gamma$-invariant and a finite subset of $G$, we can define a linear operator $T_p$ on $\cH_l^\Gamma$ by
\[
(T_p \alg)(x)=\sum_{\gamma\in \Gamma\bsl \fT_p} \alg(x\cF(\gamma^{-1})) \quad (\alg\in \cH_l^\Gamma, \;\; x\in\Q^3).
\]
Note that the operator $T_p$ commutes with $\Delta_\cQ$, because the action of $G$ commutes with $\Delta_\cQ$. 
Furthermore, $T_p$ is extended to a linear operator on $\cH_l^\Gamma\otimes \C$.

By Assumption \ref{ass2}, we have
\[
\fT_p=G\cap (K_f\, \cT_p), \quad \text{where $\cT_p=K_p\overline{\left(\begin{smallmatrix} 1&0 \\ 0 & p \end{smallmatrix}\right)}K_p$.}
\]
Take a set of representatives $\{\beta_j\}$ of $\cT_p/K_p$ and define a linear operator $\mathbf{T}_p$ on $L^2(G\bsl G_\A/ K_f)$ by 
    \[
    (\mathbf{T}_p \phi)(x)=\sum_j \phi(x\beta_j),  \qquad \phi\in L^2(G\bsl G_\A/ K_f).
    \]
Then, we see $\cL_a(T_p\alg)=\mathbf{T}_p\cL_a(\alg)$ by the definition for $\cL_a$. 
For any primes $q$ with $q\nmid \disc(\cO)$, it is obvious that $\mathbf{T}_p$ commutes with $\mathbf{T}_q$, and hence $T_p$ also commutes with $T_q$. 
Therefore, we can simultaneously diagonalize the $T_p$'s for all primes $p$ with $p\nmid \disc(\cO)$.  
If $\alg$ is a simultaneous eigenfunction, then $\alg$ is said to be a Hecke eigenform. 
\begin{lem}\label{lem:auto}
Suppose that $\alg$ is not constant and $\alg$ is a Hecke eigenform in $\cH_l^\Gamma\otimes\C$. 
Then, there exists a unique irreducible automorphic representation $\pi$ of $G_\A$ such that $\pi$ is not $1$-dimensional and $\cL_a(\alg)$ belongs to $\pi$ for any $a\in V_D$.
\end{lem}
\begin{proof}
The strong multiplicity one theorem implies the assertion, see e.g. \cite{KR}.
\end{proof}

\section{Toric periods of modular forms}\label{sec:3}

\subsection{Embedding of quadratic fields}

Let $X(D)$ be the set of quadratic fields which are embedded in $D$. 
Here, the word ``embedding'' means an injective homomorphism from $E$ to $D$ over $\Q$, that is, for each $\iota\in \Emb(E,D)$, we have $\iota(k)=k$ (for all $k\in\Q)$, $\iota(x+y)=\iota(x)+\iota(y)$, $\iota(xy)=\iota(x)\iota(y)$ (for all $x,y\in E$).
Note that any real quadratic field can not be embedded into $D$, since $D$ is definite.  
Throughout this paper, any imaginary quadratic field is regarded as a subfield of $\C$.
For $E\in X(D)$, we denote by $\Delta_E$ the fundamental discriminant of $E$. 
Then, $\Delta_E$ is a negative integer and we have $E=\Q+\Q\cdot((\Delta_E+\sqrt{\Delta_E})/2)$. 
For $E\in X(D)$,  we write the set of embeddings $E \hookrightarrow D$ by $\Emb(E,  D)$.
By the Skolem-Noether theorem,  $D^\times$ acts transitively on $\Emb(E,  D)$ by conjugation.

Take a quadratic field $E\in X(D)$ and an embedding $\iota\in\Emb(E,D)$.
If $\omega$ is the non-trivial element of the Galois group of $E$ over $\Q$, then we have $\omega(x)=\overline{x}$ where $\overline{x}$ is the complex conjugation of $x$. 
Since $E\cong\iota(E)/\Q$ and $\sigma$ is a non-trivial isomorphism over $\iota(E)$ over $\Q$, we find $\sigma(\iota(x))=\iota(\omega(x))$ for arbitrary $x\in E$. 
Hence, we have $\Nm(\iota(x))=\iota(|x|^2)$ and $\Tr(\iota(x))=\iota(x+\overline{x})$ for all $x\in E$.

\subsection{Toric periods}\label{subsec:Toric periods}

We fix $E\in X(D)$ and an embedding $\iota_0\in \Emb(E,D)$.
Let $\underline{T}$ be the subtorus of $\underline{G}$ such that $\underline{T}(R)=(\iota_0(E)\otimes_\Q R)^\times/R^\times$ for any $\Q$-algebra $R$.
We write $T$ (resp. $T_R$) for $\underline{T}(\Q)$ (resp. $\underline{T}(R)$). 
Let $\d t=\prod_v \d t_v$ be the Tamagawa measure on $T_\A$.
The toric period of $\phi \in L^2(G\bsl G_\A)$ with respect to $E$ is defined by the integral
     \[
     \cP_{\iota_0,  E}(\phi)=\int_{T\bs T_\A}\phi(t) \d t. 
     \]  
Let $\pi=\otimes_v\pi_v$ be an irreducible automorphic representation of $G_\A$, which is not a character. 
A linear form $\cP_{\iota_0,E}\colon \pi\to\C$ is defined by the integral $\cP_{\iota_0,E}(\phi)$ $(\phi\in \pi)$. 
Note that the property that $\cP_{\iota_0,  E}\not\equiv0$ on $\pi$ is independent of the choice of $\iota_0$.
Hence we write $\cP_E\not\equiv0$ on $\pi$ if $\cP_{\iota_0,  E}\not\equiv0$ on $\pi$ for some $\iota_0$.
We say that $\pi$ is $E^\times$-distinguished if $\cP_E\not\equiv 0$. 
Let $\pi'=\otimes_v\pi'_v$ be the Jacquet-Langlands transfer of $\pi$ to $\PGL_2(\A)$.
Take a function $\phi\in\pi$ with decomposition $\phi=\otimes_v\phi_v$. 
We set $E_v=E\otimes_\Q\Q_v$ for each place $v$ of $\Q$. 
Waldspurger \cite{Wal2} proved the relation between the toric period and the central $L$-value of $\pi$.
\begin{thm}[Waldspurger]\label{thm:Waldspurger}
If $\cP_E\not\equiv0$ on $\pi$,  then the ramification set of $D$ coincides with the set of places $v$ at which we have $\varepsilon(\frac12,\pi'_{E_v})=-1$, where $\varepsilon(\frac12,\pi'_{E_v})$ is the root number of the local base change of $\pi_v'$ to $\PGL_2(E_v)$.
Conversely,  if this condition is satisfied,  then $\cP_E\not\equiv0$ on $\pi$ if and only if $L(\tfrac12,  \pi')L(\tfrac12,  \pi'\otimes\eta_E)\neq0$. 
\end{thm}

\subsection{Optimal embedding and toric period}\label{sec:periodformula}

Let $\fo_E$ denote the ring of integers in $E\in X(D)$. 
We write $\Emb(\fo_E,\cO)$ for the set of embedding $\iota\in \Emb(E,D)$ such that $\iota(E)\cap \cO$ is isomorphic to $\fo_E$ over $\Z$. 
An embedding in $\Emb(\fo_E,\cO)$ is said to be optimal.

We define a lattice $L(\cO)$ in $V_D$ by $L(\cO)=\{ x\in\Z+2\cO \mid \Tr(x)=0 \}$. 

\begin{lem}\label{lem:A}
Let $\iota\in\Emb(\fo_E,\cO)$ be an embedding. 
Then, we have $\iota(\fo_E)=\iota(E)\cap \cO$, and hence there exists an element $a\in \iota(E)\cap V_D$ such that $\Nm(a)=-\Delta_E$. 
Moreover, such element $a$ belongs to $\iota(E)\cap L(\cO)$, and we have $\Z a=\iota(E)\cap L(\cO)$. 
\end{lem}
\begin{proof}
Since $\iota(E)\cap \cO$ is isomorphic to $\fo_E$ over $\Z$, there exists an element $b\in \iota(E)\cap \cO$ such that 
\[
\iota(E)\cap \cO=\Z+\Z b , \quad 
\Tr(b)=\begin{cases} 0 & \text{if $\Delta_E\equiv 0 \mod 4$}, \\ 1 & \text{if $\Delta_E\equiv 1 \mod 4$,} \end{cases}
\]
\[
\Nm(b)=\begin{cases} -\Delta_E /2 & \text{if $\Delta_E\equiv 0 \mod 4$}, \\ (1-\Delta_E)/4 & \text{if $\Delta_E\equiv 1 \mod 4$.} \end{cases}
\]
There exists an element $\alpha\in E\subset \C$ such that $b=\iota(\alpha)$.
Then we have $\Nm(b)=|\alpha|^2$, and therefore $\alpha=\pm\sqrt{\Delta_E/2}$ if $\Delta_E\equiv 0 \mod 4$, and $\alpha=(1\pm\sqrt{\Delta_E})/2$ if $\Delta_E\equiv 1 \mod 4$.
Hence, we have $\iota(\fo_E)=\iota(E)\cap \cO$, and  the element $a=b-\sigma(b)=\iota(\alpha-\overline{\alpha})$ satisfies that $a\in \iota(E)\cap V_D$ and $\Nm(a)=-\Delta_E$. 
Since $\Tr(a)=0$, we have $a\in \iota(E)\cap L(\cO)$. 
Since $\iota(E)\cap L(\cO)=\iota(\fo_E)\cap L(\cO)$, it is sufficient to prove
\[
\iota(\fo_E)\cap L(\cO)= A\qquad \text{where} \;\;  A=\{ u\in\Z+2\iota(\fo_E) \mid \Tr(u)=0\}.  
\]
The inclusion $A\subset \iota(\fo_E)\cap L(\cO)$ is obvious, hence we need to show $\iota(\fo_E)\cap L(\cO) \subset A$. 
We take $x+yb\in \iota(\fo_E)\cap L(\cO)$ $(x,y\in\Z)$. 
\begin{enumerate}
\renewcommand{\labelenumi}{(\roman{enumi})}
\item First, we assume $\Delta_E\equiv 0 \mod 4$. In this case, we have $\Tr(b)=0$, which implies $x=0$, and hence $yb\in L(\cO)$. This shows $y\in 2\Z$ and therefore $x+yb\in A$.
\item Next, we assume $\Delta_E\equiv 1 \mod 4$. Since $\Tr(x+yb)=0$, we have $y\in 2\Z$ and $2x+y=0$. Hence, $x+yb=y\iota(\sqrt{\Delta_E}/2)$, and therefore we have $x+yb\in A$.
\end{enumerate}
The assertion follows from (i) and (ii).
\end{proof}

The action of the group $G$ on $\Emb(E,D)$ is given by
\[
(g\cdot \iota)(x)=g\iota(x)g^{-1} \qquad (g\in G, \;\; \iota\in \Emb(E,D), \;\; x\in E).
\]
For $\iota_1, \iota_2 \in \Emb(\fo_E,\cO)$, if there is an element $\gamma \in \Gamma$ such that $\gamma\cdot \iota_1=\iota_2$, we denote $\iota_1\sim\iota_2$, which defines an equivalence relation on $\Emb(\fo_E,\cO)$. 
Write $\Emb(\fo_E,\cO)/\sim$ for the set of equivalence classes in $\Emb(\fo_E,\cO)$. 


Let us recall the action of the ideal class group of $\fo_E$ on the set $\Emb(\fo_E,\cO)/\sim$. 
See \cite{Voight}*{Corollary 30.4.23} for the detail. 
As explained in \cite{SWY2}*{\S3.2, \S3.3}, it is closely related to toric periods.
Take an element $E\in X(D)$, define a torus $T$ as above, and fix an optimal embedding $\iota_0\in\Emb(\fo_E,\cO)$.
Subsets $B$, $B_p$, $B_\A$ of $G$, $G_p$, $G_\A$ are respectively defined as
\[
B=\{ g\in G \mid g^{-1}\cdot\iota_0\in\Emb(\fo_E,\cO)\},
\]
\[
B_p=\{ g\in G_p \mid g^{-1}\cdot (\iota_0\otimes\mathrm{id}_{\Q_p}) \in\Emb(\fo_{E,p},\cO_p)\}, \quad B_\A=G_\inf \prod_p B_p.
\]
Note that the torus $T$ (resp. $T_\A$) is contained in $B$ (resp. $B_\A$).
\begin{lem}\label{lem:B}
A mapping $B\ni g \mapsto g^{-1}\cdot \iota_0 \in \Emb(\fo_E,\cO)$ gives a bijection $T \bsl B / \Gamma \; \cong \; \Emb(\fo_E,\cO)/\sim$. 
Furthermore, the diagonal embedding $B\to B_\A$ gives a bijection $T \bsl B / \Gamma \; \cong \; T \bsl B_\A / K$. 
\end{lem}
\begin{proof}
By the Skolem-Noether theorem, every optimal embedding $\iota$ in $\Emb(\fo_E,\cO)$ is obtained by $g^{-1}\cdot\iota_0$ for some $g\in G$.  
Hence, the mapping $B\ni g \mapsto g^{-1}\cdot \iota_0 \in \Emb(\fo_E,\cO)$ is surjective. 
Since $g\in T$ if and only if $g^{-1}\cdot \iota_0=\iota_0$ $(g\in G)$, the first assertion is proved. 
By Assumption \ref{ass2}, one has
\[
B_\A/K= (B_\A\cap GK)/K\cong (B_\A\cap G)/(G\cap K)=B/\Gamma.
\]
This shows the second assertion. 
\end{proof}

We put $U_f= \prod_{p<\inf} \overline{\iota_0(\fo_{E,p}^\times)}$, where $\fo_{E,p}=\fo_E\otimes \Z_p$, and set $U=T_\R U_f$. 
Note that $U_f= T_\A\cap K_f$. 
The ideal class group of $\fo_E$ is identified with the group $T\bsl T_{\A_f} /U_f=T\bsl T_\A / U$, and so we denote $\Cl(\fo_E)= T\bsl T_{\A} /U$. 
Hence there exist elements $b_j\in T_{\A_f}$ ($j=1,\ldots , h_K$) such that
\[
T_\A=\bigsqcup_{j=1}^{h_E} T \, b_j \, U ,
\]
where $h_E$ is the class number of $E$, i.e. $h_E=\#\Cl(\fo_E)$.
The group $\Cl(\fo_E)$ acts on $T \bsl B_\A / K$ as $t \cdot T b K = T(tb)K$ $(t\in \Cl(\fo_E)$, $TbK\in T \bsl B_\A / K)$. 
It is known that this action is well-defined and fixed point free (cf. \cite{Voight}*{Proof of Theorem 30.4.7}).
Since $T_\A$ is contained in $B_\A$, it follows from Lemma \ref{lem:B} that, for some $\gamma_j \in B$ and $k_j\in K_f$, we have
\begin{equation}\label{eq:clelem}
b_j = \gamma_j\, (\gamma_j)_\inf^{-1} \, k_j
\end{equation}
where $(\gamma_j)_\inf$ means the $\inf$-component of $\gamma_j\in G\subset G_\A$. 
It follows from Lemma \ref{lem:B} and \eqref{eq:clelem} that $Tb_jK$ corresponds to the optimal embedding $\gamma_j^{-1}\cdot \iota_0$.

Suppose $\Emb(\fo_E,\cO)\neq\emptyset$ and fix $\iota_0\in \Emb(\fo_E,\cO)$. 
By Lemma \ref{lem:A}, we can take an element $a_0\in \iota_0(E)\cap V_D$ such that $\Nm(a_0)=-\Delta_E$ and $\Z a_0= \iota_0(E)\cap L(\cO)$. 
Denote $\iota_j= \gamma_j^{-1}\cdot \iota_0 \in \Emb(\fo_E,\cO)$ and $a_j= a_0\cdot \gamma_j$. 
Obviously we have $\Nm(a_j)=-\Delta_E$ and $\Z a_j = \iota_j(E)\cap L(\cO)$.
The finite set
\begin{equation}\label{eq:orbit}
   \Orb(\iota_0)= \{ (\iota_j,a_j) \mid 1\le j\le h_E \} \quad (\subset L(\cO))    
\end{equation}
is regarded as the orbit of $\iota_0$ by $\Cl(\fo_E)$. 
The following lemma shows that the toric period is given by a finite sum over $\Orb(\iota_0)$. 
\begin{lem}\label{lem:periodformula}
There exists a constant $c>0$ such that for any $\alg\in\cH_l^\Gamma$ we have
\[
\cP_{\iota_0,E}(\cL_{a_0}(\alg))=c \, \sum_{j=1}^{h_E} \alg\circ\cT(a_j).
\]
\end{lem}
\begin{proof}
Since $\cL_{a_0}(\alg)(gk_f)=\cL_{a_0}(\alg)(g)$ (for all $g\in G_\A$, $k_f\in K_f$), we find
\[
\cP_{\iota_0,E}(\cL_{a_0}(\alg))=\vol(U_f) \int_{T\bsl T_\A/ U_f} \cL_{a_0}(\alg)(t) \, \d t.
\]
Set $W=T\cap U_f$, and then we have $T\bsl T_\A / U_f = \bigsqcup_{j=1}^{h_E} b_j\cdot (W\bsl T_\R)$. 
Hence, one obtains the formula
\[
\cP_{\iota_0,E}(\cL_{a_0}(\alg))=\frac{\vol(U_f)}{\#(W)} \sum_{j=1}^{h_E}\int_{T_\R} \cL_{a_0}(\alg)(b_j\, t_\inf) \, \d t_\inf. 
\]
By \eqref{eq:clelem} $b_j=\gamma_j(\gamma_j)_\inf^{-1}k_j$, the period integral $\cP_{\iota_0,E}(\cL_{a_0}(\alg))$ equals
\[
\frac{\vol(U_f)}{\#(W)}  \sum_{j=1}^{h_E} \int_{T_\R}\cL_{a_0}(\alg)(  (\gamma_j)_\inf^{-1}\, t_\inf ) \, \d t_\inf = \frac{\vol(U_f)}{\#(W)}  \sum_{j=1}^{h_E} \int_{T_\R} \alg( \cT(t_\inf a_0 t_\inf^{-1})\cdot \cF(\gamma_j)) \, \d t_\inf . 
\]
Since $a_0\in \iota_0(E)$ and $T$ is commutative, we have $t_\inf a_0 t_\inf^{-1}=a_0$ and therefore
\[
\cP_{\iota_0,E}(\cL_{a_0}(\alg))=\frac{\vol(U_f)\, \vol(T_\R)}{\#(W)}  \sum_{j=1}^{h_E} \alg( \cT(a_0)\cdot \cF(\gamma_j)) .
\] 
Then the assertion follows from the fact $a_j=\gamma_j^{-1}a_0\gamma_j$.
\end{proof}

\section{Jacquet--Langlands correspondence and root numbers}\label{sec:JL}

Throughout this section, we suppose $\disc(\cO)$ is square-free. 
Let $\fO$ be an Eichler order in $D$ containing $\cO$.  
Define an inner product on $\cH_l^{\overline{\fO^\times}}\otimes\C$ by the $L^2$-inner product using the identification between $\cH_l^{\overline{\fO^\times}}\otimes\C$ and $\cL_a(\cH_l^{\overline{\fO^\times}}\otimes\C)(\subset L^2(G\bsl G_\A/K_f))$. 
Hence, we can define $(\cH_l^\Gamma\otimes\C)^\new$ by the orthogonal compliment of $\sum_{\fO\supsetneq \cO}\cH_l^{\overline{\fO^\times}}\otimes\C$, where $\fO$ runs over all Eichler orders containing $\cO$.

\subsection{Jacquet--Langlands correspondence}

We review the Jacquet--Langlands correspondence between $\PGL_2$ and $G$.
For details, we refer to \cite{Shimizu}.

Let $S_k^\new(\Gamma_0(N))$ denote the space of holomorphic cusp newforms of weight $k$ for $\Gamma_0(N)$, and $T'_p$ denote the Hecke operator on $S_k^\new(\Gamma_0(N))$ for each prime $p\nmid N$.  
\begin{thm}[Eichler, Shimizu, Jacquet--Langlands]\label{thm:jl}
Let $N=\disc(\cO)$ be a square-free integer. 
There exists an isomorphism $\mathrm{JL}$ from $(\cH_l^\Gamma\otimes\C)^\new$ to $S_{2l+2}^\new(\Gamma_0(N))$ such that
\[
\mathrm{JL}(T_p\, \alg)=T'_p\, \mathrm{JL}(\alg) \quad \text{for any $\alg\in (\cH_l^\Gamma\otimes\C)^\new$ and any prime $p\nmid N$}.
\]
\end{thm}

Take a Hecke eigenform $\alg\in (\cH_l^\Gamma\otimes\C)^\new$ such that $\alg$ is not constant.
By Lemma \ref{lem:auto}, we have an irreducible automorphic representation $\pi=\otimes_v \pi_v$ of $G_\A$ such that the automorphic form $\cL_a(\alg)$ belongs to $\pi$ for any $a\in\Q^3$, 
we have $\prod_p p^{c(\pi_p)}=\disc(\cO)$ (where $c(\pi_p)$ is the conductor of $\pi_p$), and $\pi$ is not $1$-dimensional.
Under these conditions, $\pi$ satisfies the following conditions:
\begin{itemize}
\item $\pi_\inf$ is isomorphic to $\rho_l$. 
\item When $p\mid \disc(D)$, $\pi_p$ is isomorphic to the trivial representation $\trep_p$ or its twist $\omega_p\trep_p=\trep_p\otimes \omega_p$, where $\omega_p$ is the non-trivial unramified quadratic character on $\Q_p^\times$. 
\item When $p\mid \level(\cO)$, $\pi_p$ is isomorphic to the Steinberg representation $\St_p$ or its twist $\omega_p\St_p=\St_p\otimes \omega_p$. 
\item When $p\nmid \disc(\cO)$, $\pi_p$ is an unramified representation. 
\end{itemize}

The holomorphic cusp form $\mathrm{JL}(\alg)$ is also a Hecke eigenform by Theorem \ref{thm:jl}.
Let $\pi'=\otimes_v \pi'_v$ denote the irreducible automorphic representation of $\PGL_2(\A)$ which is generated by $\mathrm{JL}(\alg)$. 
Then, $\pi'$ satisfies the following:
\begin{itemize}
\item $\pi_\inf'$ is the discrete series with the minimal $\SO_2$-type is $\left[\left(\begin{smallmatrix} \cos\theta & \sin\theta \\ -\sin\theta & \cos\theta\end{smallmatrix}\right)\mapsto e^{\pm 2(l+1)i\theta}\right]$. 
\item Let $p\mid \disc(D)$. We have $\pi_p'=\St_p$ if $\pi_p=\trep_p$, and $\pi_p'=\omega_p\St_p$ if $\pi_p=\omega_p\trep_p$. 
\item For each $p\nmid \disc(D)$, we have $\pi_p'=\pi_p$. 
\end{itemize}

\subsection{Root numbers}

Take a Hecke eigenform $\alg\in (\cH_l^\Gamma\otimes\C)^\new$ such that $\alg$ is not constant, and let $\pi=\otimes_v \pi_v$ (resp. $\pi'=\otimes_v \pi_v'$) be the automorphic representation of $G_\A$ (resp. $\PGL_2(\A)$) as explained in the previous subsection. 
Let $S(\pi)$ be the finite set of places consisting of $\inf$ and primes dividing $\disc(\cO)$.  
For a quadratic field $E$, we write $\pi'_E=\otimes_v \pi'_{E,v}$ for the base change lift of $\pi'$ to $\GL_2(\A_E)$, where $\A_E$ denotes the ad\`ele ring of $E$, and write $\eps(\frac12,\pi')$ (resp. $\eps(\frac12,\pi_v')$) for the root number of $\pi'$ (resp. $\pi_v'$). 
For $\cE_v\in X(D_v)$, we write $\pi'_{\cE_v}$ for the local base change lift of $\pi_v'$ to $\GL_2(\cE_v)$.

For each element $(\cE_v)_{v\in S(\pi)}\in \prod_{v\in S(\pi)}X(D_v)$, we consider the following condition: 
\begin{cond}\label{cond:1014}
For all $ v\in S(\pi)$, $D_v$ is division if and only if $ \varepsilon(\tfrac12,\pi'_{\cE_v})=-1$. 
\end{cond}
We say that $\pi_v$ is $\cE_v^\times$-distinguished if $\Hom_{\cE_v^\times}(\pi_v,\C)\neq 0$.
It is known that, for $(\cE_v)_{v\in S(\pi)}\in \prod_{v\in S(\pi)}X(D_v)$, Condition \ref{cond:1014} holds if and only if $\pi_v$ is $\cE_v^\times$-distinguished for any $v\in S(\pi)$ (see \cite{Tunnell} and \cite{Saito}). 
Hence, if $\pi$ is $E^\times$-distinguished, then Condition \ref{cond:1014} holds for $(E_v)_{v\in S(\pi)}$.

From now on, we suppose $L(\frac12,\pi')\neq 0$, which implies $\varepsilon(\frac12,\pi')=1$. 
For each $E\in X(D)$, we denote the id\`ele character corresponding to $E$ by $\eta_E=\otimes_v\eta_{E,v}$ or $\eta=\otimes_v\eta_v$. 
It is known that
\[
\varepsilon(\tfrac12,\pi_{\cE_v}')=\varepsilon(\tfrac12,\pi_v')\, \varepsilon(\tfrac12,\pi_v'\otimes \eta_{\cE_v})\, \eta_{\cE_v}(-1)
\]
and
\begin{equation}\label{eq:eps1}
\varepsilon(\tfrac12,\pi'_E)=\varepsilon(\tfrac12,\pi')\, \varepsilon(\tfrac12,\pi'\otimes \eta)\, \eta(-1)=\varepsilon(\tfrac12,\pi'\otimes \eta).
\end{equation}
We write $\eta_{\cE_v}$ for the quadratic character on $\Q_v^\times$ corresponding to $\cE_v\in X(D_v)$.

Since $\pi_\inf'\otimes\sgn\cong\pi_\inf'$, $\cE_\inf\in X(D_\inf)=\{\C\}$, and $\varepsilon(\frac12,\pi_\inf')=-(-1)^{l/2}$, we have
\begin{equation}\label{eq:eps2}
\varepsilon(\tfrac12, \pi'_{\cE_\inf})=\varepsilon(\tfrac12,\pi_\inf')^2\, \eta_{\cE_\inf}(-1)=-1.
\end{equation}
For explicit calculations of $\eps$-factors over finite places, we refer to \cite{Schmidt}. 
Let $v$ be a finite place of $\Q$.
When $\pi_v'$ is unramified, we have $\varepsilon(\frac12,\pi_v'\otimes\eta_{\cE_v})=\eta_{\cE_v}(-1)$, and hence
\[
\varepsilon(\tfrac12, \pi'_{\cE_v})=\eta_{\cE_v}(-1)^2=1. 
\]
Let $p$ be a prime number such that $\Q_v=\Q_p$, and $\chi_p$ a quadratic character on $\Q_p^\times$. 
It is known that $\varepsilon(\tfrac12, \chi_p\St_p)=-\chi_p(p)$ if $\chi_p$ is unramified, and $\varepsilon(\tfrac12, \chi_p\St_p)=\chi_p(-1)$ if $\chi_p$ is ramified. 
Hence, for each $p\mid \disc(\cO)$, we have
\begin{equation}\label{eq:eps3}
\varepsilon(\tfrac12, \pi'_{\cE_p})=
\begin{cases}
 1 & \text{if $\pi_p'=\St_p$ and $\eta_{\cE_p}=\trep_p$}, \\ 
 -1 & \text{if $\pi_p'=\St_p$ and $\eta_{\cE_p}\neq \trep_p$}, \\ 
 -1 & \text{if $\pi_p'=\omega_p\St_p$ and $\eta_{\cE_p}= \omega_p$}, \\ 
 1 & \text{if $\pi_p'=\omega_p\St_p$ and $\eta_{\cE_p}\neq \omega_p$}.
\end{cases} 
\end{equation}
Consider the following condition for $(\cE_v)_{v\in S}\in \prod_{v\in S(\pi)}X(D_v)$:
\begin{cond}\label{cond:3}$ $
\begin{itemize}
\item When $p\mid \disc(D)$ and $\pi_p=\trep_p$, $\cE_p$ is not split over $\Q_p$. 
\item When $p\mid \disc(D)$ and $\pi_p=\omega_p\trep_p$, $\cE_p$ is an unramified extension over $\Q_p$.
\item When $p\mid \level(\cO)$ and $\pi_p=\St_p$, $\cE_p$ is split over $\Q_p$.
\item When $p\mid \level(\cO)$ and $\pi_p=\omega_p\St_p$, $\cE_p$ is a ramified extension over $\Q_p$. 
\end{itemize}    
\end{cond}
Summing up the above arguments, we obtain
\begin{prop}\label{prop:ep}
Assume $L(\frac12,\pi')\neq 0$, and take an element $(\cE_v)_{v\in S(\pi)}\in \prod_{v\in S(\pi)}X(D_v)$. 
Then the element $(\cE_v)_{v\in S(\pi)}$ satisfies Condition \ref{cond:1014} if and only if $(\cE_v)_{v\in S(\pi)}$ satisfies Condition \ref{cond:3}.
Moreover, if Condition \ref{cond:3} holds, then we have $\varepsilon(\tfrac12,\pi'\otimes\eta)=1$ for any $E\in X(D)$ such that $E_v\cong\cE_v$ for all $v\in S(\pi)$. 
\end{prop}

Let $E\in X(D)$ be an imaginary quadratic field. 
Note that, if $\pi$ is $E^\times$-distinguished, then $(E_v)_{v\in S(\pi)}$ satisfies Condition \ref{cond:1014}, and therefore it also satisfies Condition \ref{cond:3}.

By Assumption \ref{ass2}, we see that
\begin{equation}\label{eq:nop}
\displaystyle
\#(\Emb(\fo_E,\cO)/\sim)=h_E\prod_{p|\disc(D)} \left(1-\left(\frac{E}{p}\right)\right) \prod_{q|\level(\cO)} \left(1+\left(\frac{E}{q}\right)\right),
\end{equation}
where $\displaystyle \left(\frac{E}{p}\right)=1$ if $E_p$ is split over $\Q_p$, $\displaystyle \left(\frac{E}{p}\right)=-1$ if $E_p$ is an unramified extension over $\Q_p$, and $\displaystyle \left(\frac{E}{p}\right)=0$ otherwise (cf. \cite{Voight}*{Example 30.7.5}).
Therefore, we have the following proposition.
\begin{prop}\label{prop:106}
For $E\in X(D)$, 
if $(E_v)_{v\in S(\pi)}$ satisfies Condition \ref{cond:3}, then $\Emb(\fo_E,\cO)$ is not empty. 
\end{prop}

\section{Modular forms for \texorpdfstring{$\disc(\cO)=2$}{TEXT} and \texorpdfstring{$l=3$}{TEXT}}\label{sec:5}

In this section, we give a proof of Theorem \ref{thm:1} which we mentioned in the introduction.
Throughout this section, we suppose $\disc(\cO)=2$, that is, $\disc(D)=2$ and $\level(\cO) =1$. 

\subsection{$D_4$ root lattice}

A maximal order $\cO$ is given by
\[
\cO=\Z+\Z i + \Z j + \Z w , \quad i^2=j^2=-1, \;\; ij=-ji, \;\; w=\frac{1+i+j+ij}{2} .
\]
The elements in $\cO$ are called the Hurwitz integral quaternions (cf. \cite{CS}*{\S 5}), and this lattice $\cO$ is known to be the $D_4$ root lattice for the bilinear form $\langle x,y\rangle=\frac12(\Nm(x+y)-\Nm(x)-\Nm(y))$. 
The group $\Gamma$ is generated by
\[
\Gamma=\overline{\cO^\times} =\overline{\langle i,\;\; j, \;\; w  \rangle}.
\]
In addition, we have $\overline{i}^2=\overline{j}^2=\overline{w}^3=1$, and $\#\Gamma=12$.

\subsection{Basis $\bb_1$, $\bb_2$, $\bb_3$}

We choose a basis $\bb_1$, $\bb_2$, $\bb_3$ of $V_D$ as
\[
\bb_1=-i+j+ij, \qquad \bb_2=i-j+ij, \qquad \bb_3=i+j-ij. 
\]
Since $L(\cO)=\Z \, (2 i) +\Z \, (2 j) + \Z\, (i+j+ij)$, this basis $\bb_1$, $\bb_2$, $\bb_3$ forms a $\Z$-basis of $L(\cO)$, that is, $L(\cO)=\sum_{j=1}^3 \Z b_j$. 
Hence, $\cF(\Gamma)$ is a subgroup of $\SO(\cQ)\cap \SL_3(\Z)$. 
Under this choice, one has
\[
\cQ=\begin{pmatrix} 3&-1&-1 \\ -1&3&-1 \\ -1&-1&3 \end{pmatrix} \quad \text{and} \quad  \cQ^{-1}=\frac{1}{4}\begin{pmatrix} 2&1&1 \\ 1&2&1 \\ 1&1&2 \end{pmatrix} . 
\]
Hence, the Laplace operator $\Delta_\cQ$ is explicitly given by
\[
2\Delta_\cQ=\frac{\partial^2}{\partial x_1^2}+\frac{\partial^2}{\partial x_2^2}+\frac{\partial^2}{\partial x_3^2}+\frac{\partial^2}{\partial x_1\, \partial x_2}+\frac{\partial^2}{\partial x_2\, \partial x_3}+\frac{\partial^2}{\partial x_3\, \partial x_1} .
\]
Furthermore, the generators 
\[
\mathfrak{i}=\cF(\bar{i})=\begin{pmatrix} -1&-1&-1 \\ 0&0&1 \\ 0&1&0 \end{pmatrix}, \quad
\mathfrak{j}=\cF(\bar{j})=\begin{pmatrix} 0&0&1 \\ -1&-1&-1 \\ 1&0&0 \end{pmatrix}, \quad
\mathfrak{w}=\cF(\bar{w})=\begin{pmatrix} 0&0&1 \\ 1&0&0 \\ 0&1&0 \end{pmatrix}.
\]
of $\cF(\Gamma)$ satisfy the relations
\[
\mathfrak{i}^2=\mathfrak{j}^2=\mathfrak{w}^3=1,\quad \mathfrak{i}\, \mathfrak{j}=\mathfrak{j} \, \mathfrak{i} ,\quad \mathfrak{i}\, \mathfrak{w}=\mathfrak{w} \, \mathfrak{j},\quad \mathfrak{w}\, \mathfrak{i}=\mathfrak{i}\, \mathfrak{j} \, \mathfrak{w}.
\]
Therefore, $\Gamma$ is isomorphic to the alternating group of degree $4$.

\subsection{Modular forms of degree $l=3$}

Suppose $l=3$. 
A harmonic polynomial $\alg \in \cW_3$ is given by
\[
\alg(x)=x_1^3+x_2^3+x_3^3-x_1^2x_2-x_1^2x_3-x_2^2x_1-x_2^2x_3-x_3^2x_1-x_3^2x_2+2x_1x_2x_3.
\]
By direct calculations, it is easy to see that $\alg$ belongs to $\cH_3^\Gamma$. 
Since $\cH_3^\Gamma$ is isomorphic to the space of holomorphic newforms of $\Gamma_0(2)$ and weight $8$ by Theorem \ref{thm:jl}, we find $\dim \cH_3^\Gamma=1$, which means that $\cH_3^\Gamma=\Q \alg$ and $\alg$ is a Hecke eigenform in $\cH_3^\Gamma$. 
Hence, it follows from Lemma \ref{lem:auto} that there exists a unique irreducible automorphic representation $\pi=\otimes_v \pi_v$ of $G_\A$ such that $\cL_a(\alg)\in \pi$ for all $a\in V_D$.

Take a quadratic field $E\in X(D)$, and assume that $\Emb(\fo_E,\cO)$ is not empty. 
Then, by Lemma \ref{lem:A}, we can choose an optimal embedding $\iota_0$ in $\Emb(\fo_E,\cO)$, and an element $a_0\in \iota_0(E)\cap V_D$ such that $\Nm(a_0)=-\Delta_E$ and $\Z a_0=\iota_0(E)\cap L(\cO)$.  
In addition, we take the orbit $\Orb(\iota_0)=\{ (\iota_j,a_j) \mid 1\le j\le h_E\}$ as in \eqref{eq:orbit}. 
Note that they satisfy $a_j\in \iota_j(E)\cap L(\cO)$ and $\Nm(a_j)=-\Delta_E$. 
\begin{lem}\label{lem:d4sum}
Suppose that $\Delta_E\equiv 1 \mod 4$. 
Then we have $\alg\circ\cT(a_j)\equiv 1 \mod 2$ for any $j$ $(1\le j\le h_E)$, and we also see that
\[
\sum_{j=1}^{h_E}\alg\circ\cT(a_j)\equiv h_E \mod 2 .
\]
\end{lem}
\begin{proof}
Set $a_j=a_{j1}\bb_1+a_{j2}\bb_2+a_{j3}\bb_3\in\cL(\cO)$ $(a_{j1}$, $a_{j2}$, $a_{j3}\in\Z)$. 
Since $\Nm(a_j)=-\Delta_E$ we have
\begin{equation*}\label{eq:d4:1}
3a_{j1}^2+3a_{j2}^2+3a_{j3}^2-2a_{j1}a_{j2}-2a_{j2}a_{j3}-2a_{j3}a_{j1}=-\Delta_E.
\end{equation*}
From this we deduce
\begin{equation}\label{eq:d4:3}
(a_{j1}+a_{j2}+a_{j3})^2\equiv \Delta_E\equiv 1 \mod 4, \quad \text{hence} \quad  a_{j1}+a_{j2}+a_{j3} \equiv \pm 1 \mod 4 .
\end{equation}
On the other hand, we see that
\begin{equation}\label{eq:d4:2}
\alg\circ \cT(a_j)=(a_{j1}+a_{j2}+a_{j3})^3  -4(a_{j1}^2a_{j2}+a_{j1}^2a_{j3}+a_{j2}^2a_{j1}+a_{j2}^2a_{j3}+a_{j3}^2a_{j1}+a_{j3}^2a_{j2}+a_{j1}a_{j2}a_{j3})
\end{equation}
by using
\[
(a_{j1}+a_{j2}+a_{j3})^3=a_{j1}^3+a_{j2}^3+a_{j3}^3+3a_{j1}^2a_{j2}+3a_{j1}^2a_{j3}+3a_{j2}^2a_{j1}+3a_{j2}^2a_{j3}+3a_{j3}^2a_{j1}+3a_{j3}^2a_{j2}+6a_{j1}a_{j2}a_{j3}.
\]
Using \eqref{eq:d4:3} and \eqref{eq:d4:2} we have
$
\alg\circ \cT(a_j)\equiv (a_{j1}+a_{j2}+a_{j3})^3 \equiv \pm 1 \mod 4.
$
\end{proof}

As in \S\ref{sec:JL}, we write $\pi'=\otimes_v \pi'_v$ for the automorphic representation of $\PGL_2(\A)$ corresponding to $\pi$ via the Jacquet-Langlands correspondence.
The following is one of our main theorems. 
\begin{thm}\label{thm:disc2}
Let $E\in X(D)$ be an imaginary quadratic field. 
Suppose that $-\Delta_E$ is a prime number and $\Delta_E\equiv 5 \mod 8$. Then we have $L(\frac12,\pi'\otimes\eta_E)\neq 0$. 
\end{thm}
\begin{proof}
For $p=2$, we find $\pi'_p=\omega_p\St_p$ by the Atkin-Lehner sign (cf. \cite{lmfdb}). 
By $S(\pi)=\{\inf,2\}$, Propositions \ref{prop:ep} and \ref{prop:106}, the condition $\Delta_E\equiv 5 \mod 8$ implies that $(E_v)_{v\in S(\pi)}$ satisfies Condition \ref{cond:3}, $\varepsilon(\tfrac12,\pi'\otimes\eta_E)=1$, and $\Emb(\fo_E,\cO)$ is not empty.
Hence, we can apply Lemma \ref{lem:d4sum} to evaluate $\sum_{j=1}^{h_E} \alg\circ \cT(a_j)$, and the value does not vanish, since one can show that $h_E$ is odd by the Gauss' genus theory (or Dirichlet's class number formula) and the assumptions on $\Delta_E$. 
Therefore, the assertion follows from Theorem \ref{thm:Waldspurger} and Lemma \ref{lem:periodformula}. 
\end{proof}

Since $S(\pi)=\{\inf,2\}$, it follows from \eqref{eq:eps1}, \eqref{eq:eps2} and \eqref{eq:eps3} that $\eps(\frac12,\pi'\otimes\eta_E)=1$ is equivalent to $\Delta_E\equiv 5 \mod 8$. 
Hence, Theorem \ref{thm:1} in the introduction is a direct consequence of Theorem \ref{thm:disc2}. 
Note that $\Delta_E\equiv 0$, $1$, $4\mod 8$ is equivalent to $\eps(\tfrac12,\pi'\otimes\eta_E)=-1$, which implies $L(\frac12,\pi'\otimes\eta_E)=0$. 
We expect that the condition $\Delta_E\equiv 5\mod 8$ is equivalent to the condition $L(\frac12,\pi'\otimes\eta_E)\neq 0$, even if $-\Delta_E$ is not a prime number. 
However, our method is not useful if $-\Delta_E$ is not a prime number,  since $h_E$ is always even by the genus theory in such a case.

\section{Modular forms for \texorpdfstring{$\disc(\cO)=6$ and $l=1$}{TEXT}}\label{sec:6}

In this section, we prove Theorem \ref{thm:2} by using two algebraic modular forms $\alg_1$ and $\alg_2$ of $l=1$ for $(\disc(D),\level(\cO))=(2,3)$ and $(3,2)$, respectively.

\subsection{Modular forms for \texorpdfstring{$\disc(D)=2$, $\level(\cO)=3$ and $l=1$}{TEXT}}\label{sec:23}

We consider the case $\disc(D)=2$ and $\level(\cO)=3$. 
In this case, a maximal order $\cO$ is given by
\[
\cO=\Z+\Z(i-j)+\Z(i-ij)+\Z w ,\quad i^2=j^2=-1,\quad ij=-ji, \quad w=\frac{1+i+j+ij}{2}.
\]
A maximal order for any Eichler order can be provided by MAGMA computer system \cite{BCP}.
The group $\Gamma$ is generated by $\overline{w}$, and hence we have $\#\Gamma=3$.
Take a basis of $L(\cO)$ as follows: 
\[
\bb_1=2(i-j), \;\; \bb_2=2(i-ij), \;\; \bb_3=i+j+ij. 
\]
By this basis, we see that
\[
\cF( \overline{w})=\begin{pmatrix} -1&1&0 \\ -1&0&0 \\ 0&0&1 \end{pmatrix}.
\]
Hence, we can find an algebraic modular form $\alg_1$ in $\cH_1^\Gamma$ which is explicitly given by $\alg_1(x)=x_3$.
Since $\dim S_4^\new(\Gamma_0(6))=1$ (cf. LMFDB \cite{lmfdb}), we have $\dim\cH_1^\Gamma=1$ by Theorem \ref{thm:jl}. 
Therefore, the modular form $\alg_1$ generates $\cH_1^\Gamma$, that is, $\cH_1^\Gamma=\Q \alg_1$. 
By Lemma \ref{lem:auto}, we have a unique irreducible automorphic representation $\pi_1=\otimes_v \pi_{1,v}$ of $G_\A$ such that $\cL_a(\alg_1)\in \pi_1$ for all $a\in \Q^3$. 
Let $\pi_1'=\otimes_v \pi'_{1,v}$ denote the automorphic representation of $\PGL_2(\A)$ generated by $\mathrm{JL}(\alg_1)$.

Take an imaginary quadratic field $E\in X(D)$ such that $\Emb(\fo_E,\cO)$ is not empty. 
As in \eqref{eq:orbit}, we take an orbit $\Orb(\iota_0)=\{(\iota_j,a_j)\mid 1\le j\le h_E\}$ so that $a_j\in \iota_j(E)\cap L(\cO)$ and $\Nm(a_j)=-\Delta_E$. 
For $a_j=a_{j1}\bb_1+a_{j2}\bb_2+a_{j3}\bb_3\in L(\cO)$ $(a_{j1}$, $a_{j2}$, $a_{j3}\in\Z)$ we have
\begin{equation}\label{eq:23disc}
-\Delta_E=\Nm(a_j)=8a_{j1}^2+8a_{j2}^2+3a_{j3}^2+8a_{j1}a_{j2}.
\end{equation}
Hence, one has
\begin{equation}\label{eq:cong23}
\alg_1\circ \cT(a_j)^2=a_{j3}^2\equiv -3\Delta_E \mod 8.    
\end{equation}
Note that $S(\pi)=\{\inf,2,3\}$, and by \cite{lmfdb} we see that $\pi'_{1,2}=\omega_2\St_2$ and $\pi'_{1,3}=\omega_3\St_3$. 
If $\Delta_E\equiv 5\mod 8$ and $\Delta_E\not\equiv 2\mod 3$, then $(E_v)_{v\in S(\pi)}$ satisfies Condition \ref{cond:3}, therefore we see that $\varepsilon(\tfrac12,\pi'\otimes\eta_E)=1$ and $\Emb(\fo_E,\cO)$ is not empty by Propositions \ref{prop:ep} and \ref{prop:106}. 
By \eqref{eq:cong23}, when $\Delta_E\equiv 13$ or $21 \mod 24$, we have
\begin{equation}\label{eq:23mod2}
\alg_1\circ\cT(a_j)\equiv 1 \mod 2. 
\end{equation}
Thus, from \eqref{eq:23mod2}, we obtain
\begin{equation}\label{eq:2}
\sum_{j=1}^{h_E}\alg_1\circ\cT(a_j)\equiv h_E\equiv 1 \mod 2
\end{equation}
when $\Delta_E\equiv 13$ or $21 \mod 24$ and $-\Delta_E$ is a prime.

\subsection{Modular forms for \texorpdfstring{$\disc(D)=3$, $\level(\cO)=2$ and $l=1$}{TEXT}}\label{sec:32}

Let us consider the case $\disc(D)=3$ and $\level(\cO)=2$. 
Since the following discussion is the same as in \S\ref{sec:23}, we omit some explanations and use the same notations.
By MAGMA \cite{BCP}, we obtain a maximal order $\cO$ as
\[
i^2=-1, \quad j^2=-3, \quad \cO=\Z+\Z i+\Z j+\Z \frac{1+i+j+ij}{2} .
\]
Then, we have $\Gamma=\langle \overline{i} \rangle$ $(\# \Gamma=2)$, and a basis of $L(\cO)$ is as follows:
\[
\bb_1=2i, \;\; \bb_2=2j, \;\; \bb_3=i+j+ij, \;\;  L(\cO)=\Z\bb_1+\Z\bb_2+\Z\bb_3.
\]
For this basis, we see that
\[
\cF( \overline{i})=\begin{pmatrix}1&0&0 \\ 0&-1&0 \\ 1&0&-1 \end{pmatrix}.
\]
Then it is easy to see that $\alg_2$ which is given by $\alg_2(x)=2x_1+x_3$ is an algebraic modular form in $\cH_1^\Gamma$, and we have $\cH_1^\Gamma=\C \alg_2$. 
We take an imaginary quadratic field  $E\in X(D)$ and suppose that $\Emb(\fo_E,\cO)$ is not empty. 
As in \eqref{eq:orbit}, we take an orbit $\Orb(\iota_0)=\{(\iota_j,a_j)\mid 1\le j\le h_E\}$ so that $a_j\in \iota_j(E)\cap L(\cO)$ and $\Nm(a_j)=-\Delta_E$. 
For $a_j=a_{j1}\bb_1+a_{j2}\bb_2+a_{j3}\bb_3\in L(\cO)$ $(a_{j1}$, $a_{j2}$, $a_{j3}\in\Z)$, we have
\[
-\Delta_E=\Nm(a_j)=4a_{j1}^2+12a_{j2}^2+7a_{j3}^2+4a_{j1}a_{j3}+12a_{j2}a_{j3}.
\]
Therefore, we obtain a congruence relation
\[
\alg_2\circ \cT(a_j)^2=(2a_{j1}+a_{j3})^2\equiv -\Delta_E \mod 6.
\]
Let $\pi_2'=\otimes_v \pi_{2,v}'$ (resp. $\pi_2$) denote the autormophic representation of $\PGL_2(\A)$ (resp. $G_\A$) generated by $\mathrm{JL}(\alg_2)$ (resp. $\cL_a(\alg_2)$). 
Since $\pi'_{2,2}=\omega_2\St_2$ and $\pi'_{2,3}=\omega_3\St_3$ (see \cite{lmfdb}), we find that $(E_v)_{v\in S(\pi_2)}$ satisfies Condition \ref{cond:3} if $\Delta_E\not\equiv 5\mod 8$ and $\Delta_E\equiv 2\mod 3$. 
Thus, if $-\Delta_E$ is a prime and $\Delta_E\equiv 17\mod 24$, we have
\begin{equation}\label{eq:3}
\sum_{j=1}^{h_E}\alg_2\circ\cT(a_j)\equiv 1 \mod 2.
\end{equation}

\subsection{Proof of Theorem \ref{thm:2}}

Since $\dim S_4^\new(\Gamma_0(6))=1$, we obtain $\pi_1'=\pi_2'$, which is denoted by $\pi'$.  
Hence, by \eqref{eq:2}, \eqref{eq:3}, Theorem \ref{thm:Waldspurger} and Lemma \ref{lem:periodformula}, we obtain the following theorem.
\begin{thm}\label{thm:725}
Let $E\in X(D)$ be an imaginary quadratic field. 
Suppose that $\Delta_E\equiv 13$, $17$ or $21 \mod 24$, and $-\Delta_E$ is a prime number. Then we have $L(\frac12,\pi'\otimes\eta_E)\neq 0$. 
\end{thm}

Since $S(\pi)=\{\inf,2,3\}$, it follows from \eqref{eq:eps1}, \eqref{eq:eps2} and \eqref{eq:eps3} that the condition $\eps(\frac12,\pi'\otimes\eta_E)=1$ is equivalent to (i) or (ii):
\begin{itemize}
\item[(i)] $\Delta_E\equiv 5\mod 8$ and $\Delta_E\not\equiv 2\mod 3$. 
\item[(ii)] $\Delta_E\not\equiv 5\mod 8$ and $\Delta_E\equiv 2\mod 3$. 
\end{itemize}
Hence, when $\Delta_E$ is a prime number, we have (i) or (ii) if and only if $\Delta_E\equiv 13$, $17$ or $21 \mod 24$. 
Thus, Theorem \ref{thm:725} deduces Theorem \ref{thm:2}, which we mentioned in the introduction.

\section{Other cases}\label{sec:7}

In this section, we discuss other similar congruences and non-vanishing of toric periods.
Since the following discussion is similar to the previous ones, the details are omitted.

\subsection{Modular forms for \texorpdfstring{$\disc(D)=2$, $\level(\cO)=3$ and $l=2$}{TEXT}}

Recall the setting of \S\ref{sec:23}, that is, $\cO$, $\Gamma$, $\bb_1$, $\bb_2$, $\bb_3$ and so on are the same as in \S\ref{sec:23}. 
Here we discuss the case $l=2$, which is different from $l=1$ of \S\ref{sec:23}. 
By a direct calculation, we obtain the following results:
\[
\cQ=\begin{pmatrix}8&4&0 \\ 4&8&0 \\ 0&0&3 \end{pmatrix} , \quad \cQ^{-1}=\frac{1}{12}\begin{pmatrix}2&-1&0 \\ -1&2&0 \\ 0&0&4 \end{pmatrix},\quad
6 \Delta_\cQ=\frac{\partial^2}{\partial x_1^2}+\frac{\partial^2}{\partial x_2^2}+2\frac{\partial^2}{\partial x_3^2}-\frac{\partial^2}{\partial x_1\, \partial x_2}.
\]
We know $\dim \cH^\Gamma_2=1$ by Theorem \ref{thm:jl} and \cite{lmfdb}. 
Therefore, we obtain a basis $\alg$ in $\cH^\Gamma_2$ as
\[
\alg(x)=4(x_1^2+x_1x_2+x_2^2)-3x_3^2.
\]
Let $E\in X(D)$ be an imaginary quadratic field satisfying $\Emb(\fo_E,\cO)\neq \emptyset$. 
The notations $\iota_j$, $a_j$, $a_{j*}$ are the same as in \S\ref{sec:23}. 
Then, by \eqref{eq:23disc} we obtain a congruence
\begin{equation}\label{eq:23l2e1}
\alg\circ \cT(a_j)\equiv a_{j3}^2\equiv -3\Delta_E \mod 4.
\end{equation}
Let $\pi'=\otimes_v \pi_v'$ (resp. $\pi$) be the autormophic representation of $\PGL_2(\A)$ (resp. $G_\A$) generated by $\mathrm{JL}(\alg)$ (resp. $\cL_a(\alg)$). 
By the Atkin-Lehner signs (cf. \cite{lmfdb}), we find $\pi'_2=\St_2$, $\pi'_3=\omega_3\St_3$. 
Hence, $(E_v)_{v\in S(\pi)}$ satisfies Condition \ref{cond:3} if $\Delta_E\not\equiv 1\mod 8$ and $\Delta_E\not\equiv 2\mod 3$. 
Therefore, when $\Delta_E\equiv 13$ or $21 \mod 24$,
one has
\begin{equation}\label{eq:23l2e2}
\sum_{j=1}^{h_E}\alg\circ\cT(a_j)\equiv h_E \mod 4.
\end{equation}
When $-\Delta_E$ is a prime and $\Delta_E\equiv 13$ or $21 \mod 24$, we have by \eqref{eq:23l2e1} and \eqref{eq:23l2e2} 
\begin{equation}\label{eq:23l2e3}
\sum_{j=1}^{h_E}\alg\circ\cT(a_j)\equiv 1 \mod 2.    
\end{equation}

In this case, we also obtain a different non-vanishing result.
Suppose that $E\in X(D)$ satisfies $h_E \equiv 2 \mod 4$, and $\Delta_E\equiv 13$ or $21 \mod 24$. 
Then by \eqref{eq:23l2e1} and \eqref{eq:23l2e2} we get the following congruence (non-vanishing):
\begin{equation}\label{eq:23l2e4}
\sum_{j=1}^{h_E}\alg\circ\cT(a_j)\equiv 2 \mod 4.
\end{equation}
The existence of such an $E\in X(D)$ can be seen in the paper \cite{Pizer}. 
By \cite{Pizer}*{Equation (8) in Proposition 3}, we have $h_E\equiv 2\mod 4$ if $\Delta_E=-3p \equiv 21 \mod 24$ for a prime $p$.  
By \cite{Pizer}*{Equation (14) in Proposition 4}, we have $h_E\equiv 2\mod 4$ if $\Delta_E=-pq \equiv 13$ or $21 \mod 24$ for primes $p$, $q$ such that $p\equiv 1\mod 4$, $q\equiv 3\mod 4$, and ${\displaystyle \left( \frac{p}{q}\right)=-1}$.  
Therefore, by \eqref{eq:23l2e3}, \eqref{eq:23l2e4} and \cite{Pizer} we obtain the following non-vanishing result.
\begin{thm}
Let $E\in X(D)$ be an imaginary quadratic field. 
Suppose one of the following three conditions:
\begin{itemize}
    \item $\Delta_E\equiv 13$ or $21 \mod 24$, and $-\Delta_E$ is a prime number.
    \item $\Delta_E=-3p \equiv 21 \mod 24$, where $p$ is a prime. 
    \item $\Delta_E=-pq \equiv 13$ or $21 \mod 24$, where $p$ and $q$ are primes and they satisfy $p\equiv 1\mod 4$, $q\equiv 3\mod 4$, and ${\displaystyle \left( \frac{p}{q}\right)=-1}$. 
\end{itemize}
Then we obtain $L(\frac12,\pi'\otimes\eta_E)\neq 0$.
\end{thm}

\subsection{Modular forms for \texorpdfstring{$\disc(\cO)=3$ and $l=2$}{TEXT}}

Let us consider the case $\disc(\cO)=3$. 
By MAGMA \cite{BCP}, a maximal order $\cO$ is given by 
\[
\cO=\Z+\Z i+\Z w+\Z iw, \qquad i^2=-1, \quad j^2=-3, \quad  w=\frac{1+j}{2}.
\]
From this fact, it is easy to see $\Gamma=\langle  \overline{i}, \;\; \overline{ w} \rangle \cong S_3$. 
Choose a $\Z$-basis of $L(\cO)$ as  
\[
\bb_1=j, \;\; \bb_2=i+ij, \;\; \bb_3=i-ij, \;\;  L(\cO)=\Z\bb_1+\Z\bb_2+\Z\bb_3.
\]
By a direct calculation, we see that
\[
\cQ=\begin{pmatrix} 3 & 0 & 0 \\ 0& 4 & -2 \\ 0 & -2 & 4 \end{pmatrix} , \qquad  \cQ^{-1}= \frac{1}{6} \, \begin{pmatrix} 2 & 0 & 0 \\ 0& 2 & 1 \\ 0 & 1 & 2 \end{pmatrix},
\]
and therefore
\[
3 \Delta_\cQ= \frac{\partial^2}{\partial x_1^2} + \frac{\partial^2}{\partial x_2^2} + \frac{\partial^2}{\partial x_3^2} + \frac{\partial^2}{\partial x_2\, \partial x_3} .
\]
Furthermore, we calculate
\[
\cF( \overline{i})=\begin{pmatrix}-1&0&0 \\ 0&0&1 \\ 0&1&0 \end{pmatrix}, \quad \cF(\overline{w})= \begin{pmatrix} 1&0&0 \\ 0&0&1 \\ 0&-1&-1 \end{pmatrix}.
\]
By \cite{lmfdb} and Theorem \ref{thm:jl}, we have $\dim \cH_2^\Gamma=\dim S_6^\new(\Gamma_0(3))=1$. 
The algebraic modular form $\varphi$ which is defined by $\alg(x)=3x_1^2-2x_2^2-2x_3^2+2x_2x_3$ gives rise to a basis of this space.

Take an element $E\in X(D)$ and suppose $\Emb(\fo_E,\cO)\neq \emptyset$. 
We also take an orbit $\Orb(\iota_0)=\{(\iota_j,a_j)\mid 1\le j\le h_E\}$ as in \eqref{eq:orbit}, and therefore we have $a_j\in \iota_j(E)\cap L(\cO)$ and $\Nm(a_j)=-\Delta_E$. 
For $a_j=a_{j1}\bb_1+a_{j2}\bb_2+a_{j3}\bb_3\in L(\cO)$ $(a_{j1}$, $a_{j2}$, $a_{j3}\in\Z)$ we have
\[
-\Delta_E=\Nm(a)=3a_{j1}^2+4a_{j2}^2+4a_{j3}^2-4a_{j2}a_{j3}.
\]
Hence, one obtains a modulo $6$ congruence relation
\begin{equation}\label{eq:3l2e1}
\alg\circ\cT(a_j)\equiv -\Delta_E \mod 6.    
\end{equation}
We write $\pi'=\otimes_v\pi'_v$ (resp. $\pi$) for the automorphic representation of $\PGL_2(\A)$ (resp. $G_\A$) generated by $\mathrm{JL}(\alg)$ (resp. $\cL_a(\alg)$).
By the Atkin-Lehner signs (cf. \cite{lmfdb}), we see that $\pi'_3=\St_3$ in this case. 
Hence, $(E_v)_{v\in S(\pi)}$ satisfies Condition \ref{cond:3} if $\Delta_E \equiv 0$ or $2\mod 3$. 
When $-\Delta_E\equiv 1\mod 3$, we have
\begin{equation}\label{eq:3l2e2}
\sum_{j=1}^{h_E}\alg\circ\cT(a_j)\equiv h_E \mod 3
\end{equation}
by \eqref{eq:3l2e1}.
From \eqref{eq:3l2e2} we derive the following.
\begin{thm}\label{thm:3l2}
    Let $E\in X(D)$ be an imaginary quadratic field. If $\Delta_E \equiv 2\mod 3$ and $h_E\not\equiv 0\mod 3$, then $L(\frac12,\pi'\otimes\eta_E)\neq 0$. 
\end{thm}
The indivisibility of class numbers by $3$ was used to show the weak Goldfeld conjecture, see \cites{J,KL,SWY2}. 
Theorem \ref{thm:3l2} similarly suggests that $\pi'$ satisfies the weak Goldfeld conjecture.

\subsection{Modular forms for \texorpdfstring{$\disc(D)=2$, $\level(\cO)=5$ and $l=1$}{TEXT}}

Finally, we consider the case that $\disc(D)=2$ and $\level(\cO)=5$. 
Using MAGMA \cite{BCP}, one has a maximal order $\cO$ as
\[
\cO=\Z+\Z\, ij + \Z (i+2j) +\Z \frac{1+3i+j+ij}{2} , \qquad i^2=j^2=-1.
\]
We can choose a $\Z$-basis of $L(\cO)$ as
\[
\bb_1=2ij, \;\; \bb_2=2i+4j, \;\; \bb_3=3i+j+ij, \;\;  L(\cO)=\Z\bb_1+\Z\bb_2+\Z\bb_3.
\]
By a direct calculation, we see that $\Gamma=\langle \, \overline{i\, j} \, \rangle$, $\cF(\overline{ij})=\begin{pmatrix} 1&0&0 \\ 0&-1&0 \\ 1&0&-1 \end{pmatrix}$, and we have an algebraic modular form $\alg(x)=2x_1+x_3$ in $\cH_1^\Gamma$. 
By \cite{lmfdb} and Theorem \ref{thm:jl}, we have $\dim \cH_1^\Gamma=\dim S_4^\new(\Gamma_0(10))=1$, and hence $\alg$ gives a basis of $\cH_1^\Gamma$.  

For an imaginary quadratic field $E\in X(D)$ satisfying $\Emb(\fo_E,\cO)\neq \emptyset$, 
an orbit $\Orb(\iota_0)=\{(\iota_j,a_j)\mid 1\le j\le h_E\}$ is taken as in \eqref{eq:orbit}. 
Since $a_j\in \iota_j(E)\cap L(\cO)$ and $\Nm(a_j)=-\Delta_E$, we have
\[
-\Delta_E=\Nm(a_j)=4a_{j1}^2+20a_{j2}^2+11a_{j3}^2+20a_{j2}a_{j3}+4a_{j1}a_{j3}
\]
for $a_j=a_{j1}\bb_1+a_{j2}\bb_2+a_{j3}\bb_3\in L(\cO)$ $(a_{j1}$, $a_{j2}$, $a_{j3}\in\Z)$. 
Then, we find
\begin{equation}\label{eq:10l1e1}
\alg\circ \cT(a_j)^2=(2a_{j1}+a_{j3})^2\equiv -\Delta_E \mod 10.    
\end{equation}
Write $\pi'=\otimes_v\pi'_v$ (resp. $\pi$) for the automorphic representation of $\PGL_2(\A)$ (resp. $G_\A$) generated by $\mathrm{JL}(\alg)$ (resp. $\cL_a(\alg)$).
By the Atkin-Lehner signs (cf. \cite{lmfdb}), we get $\pi'_2=\St_2$ and $\pi'_5=\St_5$. 
This means that $(E_v)_{v\in S(\pi)}$ satisfies Condition \ref{cond:3} if $\Delta_E\not\equiv 1\mod 8$ and $\Delta_E\equiv \pm 1\mod 5$. 
When $\Delta_E\equiv 21$ or $29\mod 40$, it follows from \eqref{eq:10l1e1} that
\begin{equation}\label{eq:10l1e2}
\sum_{j=1}^{h_E}\alg\circ\cT(a_j)\equiv h_E \mod 2.
\end{equation}
Hence, by \eqref{eq:10l1e2} we obtain the following result.
\begin{thm}\label{thm:10l1}
    Let $E\in X(D)$ be an imaginary quadratic field. Suppose that $-\Delta_E$ is a prime number and $\Delta_E\equiv 21$ or $29\mod 40$. 
    Then, we have $L(\frac12,\pi'\otimes\eta_E)\neq 0$. 
\end{thm}

\end{document}